\documentclass{article}
\usepackage[english]{babel}
\usepackage{amsfonts,amsmath,amssymb,amsthm}
\usepackage{verbatim}
\usepackage{listings}
\usepackage{biblatex}
\usepackage{hyperref}

\author{David Monniaux\thanks{CNRS / VERIMAG} \and
Pierre Corbineau\thanks{Universit\'e Joseph Fourier / VERIMAG}}

\title{On the Generation of Positivstellensatz Witnesses in Degenerate Cases\thanks{This work was partially supported by ANR project ``ASOPT''.}}

\newcommand{\bbQ}{\mathbb{Q}}
\newcommand{\bbR}{\mathbb{R}}
\newcommand{\bbZ}{\mathbb{Z}}
\newcommand{\bbC}{\mathbb{C}}

\newcommand{\sdpcone}{\mathcal{K}}

\newcommand{\transpose}[1]{#1^T}

\newcommand{\float}[1]{\tilde{#1}}
\newcommand{\ve}[1]{\mathbf{#1}}
\newcommand{\vect}{\textrm{Vect~}}
\newcommand{\tgerman}[1]{\emph{#1}}

\newcommand{\soft}[1]{#1}

\newtheorem{theorem}{Theorem}
\newtheorem{lemma}{Lemma}

\bibliography{decision_of_polynomial_inequalities}

\begin{document}
\maketitle

\begin{abstract}
One can reduce the problem of proving that a polynomial is nonnegative, or more generally of proving that a system of polynomial inequalities has no solutions, to finding polynomials that are sums of squares of polynomials and satisfy some linear equality (\tgerman{Positivstellensatz}). This produces a \emph{witness} for the desired property, from which it is reasonably easy to obtain a formal proof of the property suitable for a proof assistant such as Coq.

The problem of finding a witness reduces to a feasibility problem in semidefinite programming, for which there exist numerical solvers. Unfortunately, this problem is in general not strictly feasible, meaning the solution can be a convex set with empty interior, in which case the numerical optimization method fails. Previously published methods thus assumed strict feasibility; we propose a workaround for this difficulty.

We implemented our method and illustrate its use with examples, including extractions of proofs to Coq.
\end{abstract}

%\category{G.1.6}{Numerical analysis}{Optimization}[Convex programming]
%\category{F.4.1}{Mathematical logic and formal languages}{Mathematical logic}[Mechanical theorem proving]
%\terms{Algorithms, Verification}

% TODO chercher ref theoremes ayant besoin lemmes positivite

\section{Introduction}
Consider the following problem: given a conjunction of polynomial equalities, and (wide and strict) polynomial inequalities, with integer or rational coefficients, decide whether this conjunction is satisfiable over $\bbR$; that is, whether one can assign real values to the variables so that the conjunction holds. A particular case is showing that a given polynomial is nonnegative.
 
The decision problem for real polynomial inequalities can be reduced to \emph{quantifier elimination}: given a formula $F$, whose atomic formulas are polynomial (in)equalities, containing quantifiers, provide another, equivalent, formula $F'$, whose atomic formulas are still polynomial (in)equalities, containing no quantifier. An algorithm for quantifier elimination over the theory of \emph{real closed fields} (roughly speaking, $(\bbR,0,1,+,\times,\leq)$ was first proposed by Tarski \cite{Tarski51,Seidenberg54}, but this algorithm had non-elementary complexity and thus was impractical. Later, the \emph{cylindrical algebraic decomposition} (CAD) algorithm was proposed \cite{Collins75}, with a doubly exponential complexity, but despite improvements \cite{Collins98} CAD is still slow in practice and there are few implementations available.

Quantifier elimination is not the only decision method.
Basu et al. \cite[Theorem~3]{BasuPollackRoy96} proposed a satisfiability testing algorithm with complexity $s^{k+1} d^{O(k)}$, where $s$ is the number of distinct polynomials appearing in the formula, $d$ is their maximal degree, and $k$ is the number of variables. We know of no implementation of that algorithm.
Tiwari \cite{Tiwari05:CSL} proposed an algorithm based on rewriting systems that is supposed to answer in reasonable time when a conjunction of polynomial inequalities has no solution.

Many of the algebraic algorithms are complex, which leads to complex implementations. This poses a methodology problem: can one trust their results? The use of computer programs for proving lemmas used in mathematical theorems was criticized in the case of Thomas Hales' proof of the Kepler conjecture. Similarly, the use of complex decision procedures (as in the proof assistant PVS%
\footnote{\url{http://pvs.csl.sri.com/}}) or program analyzers (as, for instance, Astr\'ee%
\footnote{\url{http://www.astree.ens.fr/}}) in order to prove the correctness of critical computer programs is criticized on grounds that these verification systems could themselves contain bugs.
% TODO: should I use citations and not only urls?

One could formally prove correct the implementation of the decision procedure using a proof assistant such as Coq \cite{Coq_8.1,Mahboubi07:_implem_cad_coq}; but this is likely to be long and difficult. An alternative is to arrange for the procedure to provide a \emph{witness} of its result. The answer of the procedure is correct if the witness is correct, and correctness of the witness can be checked by a much simpler procedure, which can be proved correct much more easily.

Unsatisfiability witnesses for systems of complex equalities or \emph{linear} rational inequalities are already used within DPLL($T$) satisfiability modulo theory decision procedures~\cite[ch.~11]{Kroening_Strichman_08}\cite{SRI-CSL-06-01}. It is therefore tempting to seek unsatisfiability witnesses for systems of polynomial inequalities.

In recent years, it was suggested \cite{Parrilo_PhD} to use numerical \emph{semidefinite programming} to look for proof witnesses whose existence is guaranteed by a \tgerman{Positivstellensatz}~\cite{Krivine_64_MR0175937,Stengle73,Schweighofer2004529}. The original problem of proving that a system of polynomial inequalities has no solution is reduced to: given polynomials $P_i$ and $R$, derived from those in the original inequalities, find polynomials $Q_i$ that are sums of squares such that $\sum_i P_i Q_i = R$. Assuming some bounds on the degrees of $Q_i$, this problem is in turn reduced to a \emph{semidefinite programming} pure feasibility problem~\cite{cvxopt,Semidefinite_Programming_96}, a form of convex optimization. The polynomials $Q_i$ then form a witness, from which a machine-checkable formal proof, suitable for tools such as Coq~\cite{Coq_8.1} or Isabelle \cite{harrison-sos}, may be constructed.

Unfortunately, this method suffers from a caveat: it applies only under a \emph{strict feasibility} condition \cite{PeyrlParrilo08}: a certain convex geometrical object should not be degenerate, that is, it should have nonempty interior. Unfortunately it is very easy to obtain problems where this condition is not true. Equivalently, the method of rationalization of certificates \cite{Kaltofen_et_al_ISSAC08} has a limiting requirement that the rationalized moment matrix remains positive semidefinite.

In this article, we explain how to work around the degeneracy problem: we propose a method to look for rational solutions to a general SDP feasibility problem. We have implemented our method and applied it to some examples from the literature on positive polynomials, and to examples that previously published techniques failed to process.

\section{Witnesses}
For many interesting theories, it is trivial to check that a given valuation of the variables satisfies a quantifier-free formula. A satisfiability decision procedure will in this case tend to seek a \emph{satisfiability witness} and provide it to the user when giving a positive answer.

In contrast, if the answer is that the problem is not satisfiable, the user has to trust the output of the satisfiability testing algorithm, the informal meaning of which is ``I looked carefully everywhere and did not find a solution.'' In some cases, it is possible to provide \emph{unsatisfiability witnesses}: solutions to some form of dual or auxiliary problem that show that the original problem had no solution.

\subsection{Nonnegativity Witnesses}
\label{part:nonnegativity}
To prove that a polynomial $P$ is nonnegative, one simple method is to express it as a sum of squares of polynomials. One good point is that the degree of the polynomials involved in this sum of squares can be bounded, and even that the choice of possible monomials is constrained by the Newton polytope of $P$, as seen in \S\ref{part:solve_sos}.

Yet, there exist nonnegative polynomials that cannot be expressed as sums of squares, for instance this example due to Motzkin~\cite{Reznick96_Hil17}:
\begin{equation}\label{eqn:motzkin}
M = x_1^6 + x_2^4 x_3^2 + x_2^2 x_3^4 - 3 x_1^2 x_2^2 x_3^2
\end{equation}

However, Artin's answer to Hilbert's seventeenth problem is that any nonnegative polynomial can be expressed as a sum of squares of rational functions.%
\footnote{There exists a theoretical exact algorithm for computing such a decomposition for homogeneous polynomials of at most 3 variables~\cite{DeKlerk_Pasechnik_2004}; we know of no implementation of it and no result about its practical usability.}

It follows that such a polynomial can always be expressed as the quotient $Q_2/Q_1$ of two sums of squares of polynomials, which forms the nonnegativity witness, and can be obtained by solving $P.Q_1 - Q_2 = 0$ for $Q_1 \neq 0$ (this result is also a corollary of Th.~\ref{th:real-nullstellensatz}).

\subsection{Unsatisfiability Witnesses for Polynomial Inequalities}
\label{part:positivstellensatz}
For the sake of simplicity, we shall restrict ourselves to wide inequalities (the extension to mixed wide/strict inequalities is possible).
Let us first remark that the problem of testing whether a set of wide inequalities with coefficients in a subfield $K$ of the real numbers is satisfiable over the real numbers is equivalent to the problem of testing whether a set of \emph{equalities} with coefficients $K$ is satisfiable over the real numbers: for each inequality $P(x_1,\dots,x_m) \geq 0$, replace it by $P(x_1,\dots,x_m) - \mu^2 =0$, where $\mu$ is a new variable.
Strict inequalities can also be simulated as follows: $P_i(x_1,\dots,x_m) \neq 0$ is replaced by $P_i(x_1,\dots,x_m).\mu = 1$ where $\mu$ is a new variable.
One therefore does not gain theoretical simplicity by restricting oneself to equalities.

Stengle \cite{Stengle73} proved two theorems regarding the solution sets of systems of polynomial equalities and inequalities over the reals (or, more generally, over real closed fields): a \tgerman{Nullstellensatz} and a \tgerman{Positivstellensatz}; a similar result was proved by Krivine~\cite{Krivine_64_MR0175937}. Without going into overly complex notations, let us state consequences of these theorems.

Let $K$ be an ordered field (such as~$\bbQ$) and $K'$ be a real closed field containing $K$ (such as the real field~$\bbR$), and let $\mathbf{X}$ be a list of variables $X_1,\dots,X_n$.
$A^{*2}$ denotes the squares of elements of $A$. The \emph{multiplicative monoid} generated by $A$ is the set of products of zero of more elements from $A$. The \emph{ideal} generated by $A$ is the set of sums of products of the form $PQ$ where $Q \in K[\mathbf{X}]$ and $P \in A$. The \emph{positive cone} generated by $A$ is the set of sums of products of the form $p.P.Q^2$ where $p \in K$, $p > 0$, $P$ is in the multiplicative monoid generated by $A$, and $Q \in K[\mathbf{X}]$. Remark that we can restrict $P$ to be in the set of products of elements of $A$ where no element is taken twice, with no loss of generality.

The result \cite{Lombardi_LNM92,Lombardi_zeros_effectifs,Coste_Lombardi_Roy_2001} of interest to us is:
\begin{theorem}\label{th:real-nullstellensatz}
Let $F_>$, $F_{\geq}$, $F_=$, $F_{\neq}$ be sets of polynomials in $K[\mathbf{X}]$, to which we impose respective sign conditions $> 0$, $\geq 0$, $= 0$, $\neq 0$. The resulting system is unsatisfiable over ${K'}^n$ if and only if there exist an equality in $K[\mathbf{X}]$ of the type $S+P+Z=0$, with $S$ in the multiplicative monoid generated by $F_> \cup F_{\neq}^{*2}$ , $P$ belongs to the positive cone generated by $F_{\geq} \cup F_>$, and $Z$ belongs to the ideal generated by $F_=$.
\end{theorem}
$(S,P,Z)$ then constitute a \emph{witness} of the unsatisfiability of the system.
\footnote{Another result, due to Schm\"udgen\cite{Schweighofer2004529}, gives simpler witnesses for $P_1 \geq 0 \land \dots \land P_n \geq 0 \Rightarrow C$ in the case where $P_1 \geq 0 \land \dots \land P_n \geq 0$ defines a compact set.}

\begin{comment}
Note that this result resembles the one used for linear inequalities (Section~\ref{part:linear_ineq_witnesses}), replacing nonnegative numbers by sums of squares of polynomials.
\end{comment}

For a simple example, consider the following system, which obviously has no solution:
\begin{equation}\label{eqn:very_simple_example}
\left\{
\begin{array}{l}
-2 + y^2 \geq 0\\
1 - y^4 \geq 0
\end{array}\right.\end{equation}
A \tgerman{Positivstellensatz} witness is $y^2(-2+y^2)+1(1-y^4)+2y^2+1=0$. Another is $\left(\frac{2}{3}+\frac{y^2}{3}\right)(-2+y^2)+\frac{1}{3}(1-y^4)+1=0$.

Consider the conjunction $C$: $P_1 \geq 0 \wedge \dots \wedge P_n \geq 0$ where $P_i \in \bbQ[X_1, \dots, X_m]$. Consider the set $\Pi$ of products of the form $\prod_i P_i^{w_i}$ for $\ve{w} \in \{0,1\}^n$ --- that is, the set of all products of the $P_i$ where each $P_i$ appears at most once. Obviously, if one can exhibit nonnegative functions $Q_j$ such that $\sum_{T_j \in \Pi} Q_j T_j + 1 =0$, then $C$ does not have solutions.
Theorem~\ref{th:real-nullstellensatz} guarantees that if $C$ has no solutions, then such functions $Q_j$ exist as sum of squares of polynomials (we simply apply the theorem with $F_>=F_{\neq}=\emptyset$ and thus $S=\{1\}$). We have again reduced our problem to the following problem: given polynomials $T_j$ and $R$, find sums-of-squares polynomials $Q_j$ such that $\sum_j Q_j T_j = R$. Because of the high cost of enumerating all products of the form $\prod_i P_i^{w_i}$, we have first looked for witnesses of the form $\sum_{T_j \in S} Q_j P_j + 1 =0$.

\section{Solving the Sums-of-Squares Problem}
\label{part:solve_sos}
In \S\ref{part:nonnegativity} and \S\ref{part:positivstellensatz}, we have reduced our problems to: given polynomials $(P_j)_{1 \leq j \leq n}$ and $R$ in $\bbQ[X_1, \dots, X_m]$, find polynomials that are sums of squares $Q_j$ such that
\begin{equation}
\label{eqn:sos-equation}
\sum_j P_j Q_j = R
\end{equation}
We wish to output the $Q_j$ as $Q_j = \sum_{i=1}^{n_j} \alpha_{ji} L_{ji}^2$ where $\alpha_{ji} \in \bbQ^+$ and $L_{ji}$ are polynomials over~$\bbQ$.
We now show how to solve this equation.

\subsection{Reduction to Semidefinite Programming}
\label{part:reduction}

\begin{lemma}\label{lem:sos_is_sym_matrix}
Let $P \in K[X, Y, \dots]$ be a sum of squares of polynomials $\sum_i P_i^2$. Let $M = \{m_1, \dots, m_{|M|}\}$ be a set such that each $P_i$ can be written as a linear combination of elements of~$M$ ($M$ can be for instance the set of monomials in the $P_i$). Then there exists a $|M| \times |M|$  symmetric positive semidefinite matrix $Q$ with coefficients in $K$ such that $P(X, Y, \dots) = [m_1, \dots, m_{|M|}] Q [m_1, \dots, m_{|M|}]^T$, noting $v^T$ the transpose of~$v$.
\end{lemma}

Assume that we know the $M_j$, but we do not know the matrices $\hat{Q}_j$. The equality $\sum_j P_j (M_j \hat{Q}_j(M_j)^T) = R$ directly translates into a system $(S)$ of affine linear equalities over the coefficients of the $\hat{Q}_j$: $\sum_j (M_j \hat{Q}_j(M_j)^T) P_j - R$ is the zero polynomial, so its coefficients, which are affine linear combinations of the coefficients of the $\hat{Q}_j$ matrices, should be zero; each of these combinations thus yields an affine linear equation. The additional requirement is that the $\hat{Q}_j$ are positive semidefinite.

One can equivalently express the problem by grouping these matrices into a block diagonal matrix $\hat{Q}$ and express the system $(S)$ of affine linear equations over the coefficients of $\hat{Q}$. By exact rational linear arithmetic, we can obtain a system of generators for the solution set of $(S)$: $\hat{Q} \in -F_0 + \textrm{vect}(F_1, \dots, F_m)$. The problem is then to find a positive semidefinite matrix within this \emph{search space}; that is, find $\alpha_1,\dots,\alpha_m$ such that $-F_0+\sum_i \alpha_i F_i \succeq 0$.
This is the problem of \emph{semidefinite programming}: finding a positive semidefinite matrix within an affine linear variety of symmetric matrices, optionally optimizing a linear form \cite{Semidefinite_Programming_96,cvxopt}.

For instance, the second unsatisfiability witness we gave for constraint system~\ref{eqn:very_simple_example} is defined, using monomials $\{1,y\}$, $1$ and $\{1,y\}$, by:
\begin{equation*}
\left(\begin{array}{c|c|c}
\begin{array}{cc}
\frac{2}{3}& 0\\
0 & \frac{1}{3}\\
\end{array}&&\\
\hline
& \frac{1}{3} &\\
\hline
&&
\begin{array}{cc}
0 & 0 \\
0 & 0 \\
\end{array}
\end{array}\right)
\end{equation*}

It looks like finding a solution to Equ.~\ref{eqn:sos-equation} just amounts to a SDP problem. There are, however, several problems to this approach:
\begin{enumerate}
\item For the general \tgerman{Positivstellensatz} witness problem, the set of polynomials to consider is exponential in the number of inequalities.
\item Except for the simple problem of proving that a given polynomial is a sum of squares, we do not know the degree of the $Q_j$ in advance, so we cannot
\footnote{There exist non-elementary bounds on the degree of the monomials needed \cite{Lombardi_LNM92}. In the case of Schm\"udgen's result on compact sets, there are better bounds~\cite{Schweighofer2004529}.} choose finite sets of monomials $M_j$.
The dimension of the vector space for $Q_j$ grows quadratically in $|M_j|$.
\item Some SDP algorithms can fail to converge if the problem is not \emph{strictly feasible} --- that is, the solution set has empty interior, or, equivalently, is not full dimensional (that is, it is included within a strict subspace of the search space).
\item SDP algorithms are implemented in floating-point. If the solution space is not full dimensional, they tend to provide solutions $\hat{Q}$ that are ``almost'' positive semidefinite (all eigenvalues greater than $-\epsilon$ for some small positive $\epsilon$), but not positive semidefinite.
\end{enumerate}

Regarding problem~1, bounds on degrees only matter for the \emph{completeness} of the refutation method: we are guaranteed to find the certificate if we look in a large enough space. They are not needed for \emph{soundness}: if we find a correct certificate by looking in a portion of the huge search space, then that certificate is correct regardless. This means that we can limit the choice of monomials in~$M_j$ and hope for the best.

Regarding the second and third problems~: what is needed is a way to reduce the dimension of the search space, ideally up to the point that the solution set is full dimensional. As recalled by \cite{PeyrlParrilo08}, in a sum-of-square decomposition of a polynomial $P$, only monomials $x_1^{\alpha_1} \dots x_n^{\alpha_n}$ such that $2(\alpha_1,\dots,\alpha_n)$ lies within the Newton polytope%
\footnote{The Newton polytope of a polynomial $P$, or in Reznick's terminology, its \emph{cage}, is the convex hull of the vertices $(\alpha_1,\dots,\alpha_n)$ such that $x_1^{\alpha_1} \dots x_n^{\alpha_n}$ is a monomial of~$P$.}
of~$P$ can appear \cite[Th.~1]{Reznick78}.
This helps reduce the dimension if $P$ is known in advance (as in a sum-of-squares decomposition to prove positivity) but does not help for more general equations.

Kaltofen et al. \cite{Kaltofen_et_al_JSC09} suggest solving the SDP problem numerically and looking for rows with very small values, which indicate useless monomials that can be safely removed from the basis; in other words, they detect ``approximate kernel vectors'' from the canonical basis. Our method is somehow a generalization of theirs: we detect kernel vectors whether or not they are from the canonical basis.

In the next section, we shall investigate the fourth problem: how to deal with solution sets with empty interior.

\subsection{How to Deal with Degenerate Cases}
In the preceding section, we have shown how to reduce the problem of finding unsatisfiability witnesses to a SDP feasibility problem, but pointed out one crucial difficulty: the possible degeneracy of the solution set. In this section, we explain more about this difficulty and how to work around it. 

Let $\sdpcone$ be the cone of positive semidefinite matrices. We denote by $M \succeq 0$ a positive semidefinite matrix~$M$, by $M \succ 0$ a positive definite matrix~$M$. The vector $\ve{y}$ is decomposed into its coordinates $y_i$. $\tilde{x}$ denotes a floating-point value close to an ideal real value~$x$.

We consider a SDP feasibility problem: given a family of symmetric matrices $F_0, F_i, \dots, F_m$, find $(y_i)_{1 \leq i \leq m}$ such that
\begin{equation}\label{eq:SDP_problem}
F(\ve{y})=-F_0+\sum_{i=1}^m y_i F_i \succeq 0.
\end{equation}
The $F_i$ have rational coefficients, and we suppose that there is at least one rational solution for $\ve{y}$ such that $F(\ve{y}) \succeq 0$. The problem is how to find such a solution.

If nonempty, the solution set $S \subseteq \bbR^m$ for the $\ve{y}$, also known as the \emph{spectrahedron}, is semialgebraic, convex and closed; its boundary consists in $\ve{y}$ defining singular positive semidefinite matrices, its interior are positive definite matrices.
We say that the problem is strictly feasible if the solution set has nonempty interior. Equivalently, this means that the convex $S$ has dimension~$m$.

Interior point methods used for semidefinite feasibility, when the solution set has nonempty interior, tend to find a solution $\ve{\tilde{y}}$ in the interior away from the boundary. Mathematically speaking, if $\ve{\tilde{y}}$ is a numerical solution in the interior of the solution set, then there is $\epsilon > 0$ such that for any $\ve{y}$ such that $\|\ve{y} - \ve{\tilde{y}}\| \leq \epsilon$, $\ve{y}$ is also a solution.  Choose a very close rational approximation $\ve{y}$ of $\ve{\tilde{y}}$, then unless we are unlucky (the problem is almost degenerate and all any suitable $\epsilon$ is extremely small), then $\ve{y}$ is also in the interior of $S$. Thus, $F(\ve{y})$ is a solution of problem~\ref{eq:SDP_problem}.

 This is why earlier works on sums-of-square methods \cite{PeyrlParrilo08} have proposed finding rational solutions only when the SDP problem is strictly feasible. In this article, we explain how to do away with the strict feasibility clause.

Some problems are not strictly feasible. Geometrically, this means that the linear affine space $\{ -F_0 + \sum_{i=1}^m y_i F_i \mid (y_1, \dots, y_m) \in \bbR^m \}$ is tangent to the semidefinite positive cone~$\sdpcone$. Alternatively, this means that the solution set is included in a strict linear affine subspace of~$\bbR^m$. Intuitively, this means that we are searching for the solution in ``too large a space''; for instance, if $m=2$ and $\ve{y}$ lies in a plane, this happens if the solution set is a point or a segment of a line. In this case, some SDP algorithms may fail to converge if the problem is not strictly feasible, and those that converge, in general, will find a point slightly outside the solution set. The main contribution of this article is a workaround for this problem.

\subsection{Simplified algorithm}
\label{part:algorithm}
 We shall thus now suppose the problem has empty interior.

The following result is crucial but easily proved:
\begin{lemma}\label{lem:sdpcone_tangent_kernel}
Let $E$ be a linear affine subspace of the $n \times n$ symmetric matrices such that $E \cap \sdpcone \neq \emptyset$.
$F$ in the relative interior $I$ of $E \cap \sdpcone$. Then it follows:
\begin{enumerate}
\item For all $F' \in E \cap \sdpcone$, $\ker F \subseteq \ker F'$.
\item The least affine space containing $E \cap \sdpcone$ is $H = \{ M \in E \mid \ker M \supseteq \ker F \}$.
\end{enumerate}

\end{lemma}
Suppose we have found a numerical solution $\ve{\float{y}}$, but it is nearly singular --- meaning that it has some negative eigenvalues extremely close to zero. This means there is $\ve{v} \neq \ve{0}$ such that $|\ve{v}.F(\ve{\float{y}})| \leq \epsilon \| \ve{v} \|$. Suppose that $\ve{\float{y}}$ is very close to a rational solution $\ve{y}$ and, that $\ve{v}.F(\ve{y}) = 0$, and also that $\ve{y}$ is in the relative interior of $S$ --- that is, the interior of that set relative to the least linear affine space containing~$S$. Then, by lemma~\ref{lem:sdpcone_tangent_kernel}, all solutions $F(\ve{y'})$ also satisfy $\ve{v}.F(\ve{y'})=0$. Remark that the same lemma implies that either there is no rational solution in the relative interior, or that rational solutions are dense in~$S$.

How can finding such a $\ve{v}$ help us? Obviously, if $\ve{v} \in \bigcap_{i=0}^m \ker F_i$, its discovery does not provide any more information than already present in the linear affine system $-F_0+\vect(F_1,\dots,F_m)$. We thus need to look for a vector outside that intersection of kernels; then, knowing such a vector will enable us to reduce the dimension of the search space from $m$ to $m' < m$.

Thus, we look for such a vector in the orthogonal complement of $\bigcap_{i=0}^m \ker F_i$, which is the vector space generated by the rows of the symmetric matrices $F_0, \dots, F_m$. We therefore compute a full rank matrix $B$ whose rows span the exact same space; this can be achieved by echelonizing a matrix obtained by stacking $F_0,\dots,F_m$. Then, $\ve{v} = \ve{w}B$ for some vector $\ve{w}$. We thus look for $\ve{w}$ such that $G(\ve{y}).\ve{w}=0$, with $G(\ve{y}) = B F(\ve{y}) \transpose{B}$.

The question is how to find such a $\ve{w}$ with rational or, equivalently, integer coefficients. Another issue is that this vector should be ``reasonable'' --- it should not involve extremely large coefficients, which would basically amplify the floating-point inaccuracies.

We can reformulate the problem as: find $\ve{w} \in \bbZ^m \setminus \{0\}$ such that both $\ve{w}$ and $G(\ve{\float{y}}).\ve{w}$ are ``small'', two constraints which can be combined into a single objective to be minimized $\alpha^2 \| G(\ve{\float{y}}).\ve{w} \|_2^2 + \|\ve{w}\|_2^2$, where $\alpha > 0$ is a coefficient for tuning how much we penalize large values of $\| G(\ve{\float{y}}).\ve{w} \|_2$ in comparison to large values of $\| \ve{w} \|_2$. If $\alpha$ is large enough, the difference between $\alpha G(\ve{\float{y}})$ and its integer rounding $M$ is small. We currently choose $\alpha = \alpha_0 / \| G(\ve{\float{y}}) \|$, with $\| M \|$ the Frobenius norm of $M$ (the Euclidean norm for $n \times n$ matrices being considered as vectors in $\bbR^{n^2}$), and $\alpha_0=10^{15}$.

We therefore try searching for a small (with respect to the Euclidean norm) nonzero vector that is an integer linear combination of the $l_i = (0, \dots, 1, \dots, 0, m_i)$ where $m_i$ is the $i$-th row of~$M$ and the $1$ is at the $i$-th position. Note that, because of the diagonal of ones, the $l_i$ form a free family.

This problem is known as finding a short vector in an integer lattice, and can be solved by the Lenstra-Lenstra-Lov\'asz (LLL) algorithm. This algorithm outputs a free family of vectors $s_i$ such that $s_1$ is very short. Other vectors in the family may also be very short.

Once we have such a small vector $\ve{w}$, using exact rational linear algebra, we can compute $F'_0, \dots, F'_{m'}$ such that
{\small
\begin{multline}
\label{eq:dimension_reduction}
\left\{ -F'_0 + \sum_{i=1}^{m'} y'_i F'_i \mid (y_1, \dots, y_{m'}) \in \bbR^{m'} \right\} =\\
\left\{ -F_0 + \sum_{i=1}^{m} y_i F_i \mid (y_1, \dots, y_{m}) \in \bbR^{m} \right\} \cap
\{ F \mid F.\ve{v} = 0 \}
\end{multline}}
The resulting system has lower search space dimension $m' < m$, yet the same solution set dimension. By iterating the method, we eventually reach a search space dimension equal to the dimension of the solution set.

If we find no solution $F'_0$, then it means that the original problem had no solution (the Positivstellensatz problem has no solution, or the monomial bases were too small), or that a bad vector $\ve{v}$ was chosen due to lack of numerical precision. This is the only bad possible outcome of our algorithm: it may fail to find a solution that actually exists; in our experience, this happens only on larger problems (search space of dimension 3000 and more), where the result is sensitive to numerical roundoff. In contrast, our algorithm may never provide a wrong result, since it checks for correctness in a final phase.

\subsection{More Efficient Algorithm}
In lieu of performing numerical SDP solving on $F = -F_0 + \sum y_i F_i \succeq 0$, we can perform it in lower dimension on $-(B F_0 \transpose{B})+ \sum y_i (B F_i \transpose{B}) \succeq 0$. Recall that the rows of $B$ span the orthogonal complement of $\bigcap_{i=0}^m \ker F_i$, which is necessarily included in $\ker F$; we are therefore just leaving out dimensions that always provide null eigenvalues.

The reduction of the sums-of-squares problem (Eq.~\ref{eqn:sos-equation}) provides matrices with a fixed block structure, one block for each $P_j$: for a given problem all matrices $F_0, F_1, \dots, F_m$ are block diagonal with respect to that structure. We therefore perform the test for positive semidefiniteness of the proposed $F(\ve{y})$ solution block-wise (see Sec.~\ref{part:implementation} for algorithms). For the blocks not found to be positive semidefinite, the corresponding blocks of the matrices $B$ and $F(\ve{\float{y}})$ are computed, and LLL is performed.

As described so far, only a single $\ve{v}$ kernel vector would be supplied by LLL for each block not found to be positive semidefinite. In practice, this tends to lead to too many iterations of the main loop: the dimension of the search space does not decrease quickly enough. We instead always take the first vector $\ve{v}^{(1)}$ of the LLL-reduced basis, then accept following vectors $\ve{v}^{(i)}$ if $\| \ve{v}^{(i)} \|_1 \leq \beta.\| \ve{v}^{(1)} \|_1$ and $\| G(\float{\ve{y}}) . \ve{v}^{(i)} \|_2 \leq \gamma.\| G(\float{\ve{y}}) . \ve{v}^{(1)} \|_2$. For practical uses, we took $\beta = \gamma = 10$.

When looking for the next iteration $\float{\ve{y'}}$, we use the $\float{\ve{y}}$ from the previous iteration as a hint: instead of starting the SDP search from an arbitrary point, we start it near the solution found by the previous iteration. We perform least-square minimization so that $-F'_0 + \sum_{i=1}^{m'} y'_i F'_i$ is the best approximation of $-F_0 + \sum_{i=1}^{m} y_i F_i \mid (y_1, \dots, y_{m})$.

\subsection{Extensions and Alternative Implementation}
As seen in \S\ref{part:examples}, our algorithm tends to produce solutions with large numerators and denominators in the sum-of-square decomposition. We experimented with methods to get $F(\ve{y'}) \approx F(\ve{y})$ such that $F(\ve{y'})$ has a smaller common denominator. This reduces to the following problem: given $\ve{v} \in \ve{f}_0 + \textrm{vect}(\ve{f}_1, \dots, \ve{f}_n)$ a real (floating-point) vector and $\ve{f}_0, \dots, \ve{f}_n$ rational vectors, find $\ve{y}'_1,\dots,\ve{y}_n$ such that $\ve{v}' = \ve{f}_0 + \sum_i y'_i \ve{f}_i \approx \ve{v}$ and the numerators of $\ve{v}'$ have a tunable magnitude (parameter $\mu$). One can obtain such a result by LLL reduction of the rows of:
\begin{equation} M =
\begin{pmatrix}
\bbZ(\beta\mu(\ve{f}_0 - \ve{v})) & \bbZ(\beta \ve{f}_0) & 1 & 0 & \dots & 0 \\
\bbZ(\beta\mu \ve{f}_1)      & \bbZ(\beta \ve{f}_1) & 0 & 1 & \dots &   \\
\vdots                  & \vdots          & 0 &   & \ddots & \\
\bbZ(\beta\mu \ve{f}_n)      & \bbZ(\beta \ve{f}_n) & 0 &   &       & 1\\
\end{pmatrix}
\end{equation}
where $\beta$ is a large parameter (say, $10^{19}$) and $\bbZ(v)$ stands for the integer rounding of $v$. After LLL reduction, one of the short vectors in the basis will be a combination $\sum_{i}^n y_i l_i$ where $l_0, \dots, l_n$ are the rows of $M$, such that $y_0 \neq 0$. Because of the large $\beta\mu$ coefficient, $y_0(\ve{f}_0-\ve{v}) + \sum_{i=1}^n y_i \ve{f}_i$ should be very small, thus $f_0 + \sum_{i=1}^n y_i \ve{f}_i \approx v$. But among those vectors, the algorithm chooses one such that $\sum_{i=0}^n y_i \ve{f}_i$ is not large --- and among the suitable $v'$, the vector of numerators is proportional to $\sum_{i=0}^n y_i \ve{f}_i$.

After computing such a $y'$, we check whether $F(\ve{y'}) \succeq 0$; we try this for a geometrically increasing sequence of $\mu$ and stop as soon as we find a solution.
The matrices $\hat{Q}_j$ then have simpler coefficients than the original ones. Unfortunately, it does not ensue that the sums of square decompositions of these matrices have small coefficients.

\begin{comment}
We have so far considered SDP feasibility problems only. It is also possible to use the same method for optimization problems. First, the feasibility problem should be solved by the above method; a parametrization $-F'_0 + \textrm{vect}(F'_1, \dots, F'_m)$ of the linear affine span of the solution set is thus computed, as well as a point $-F'_0 + \sum_i y'_i F'_i \succeq 0$. In this new parametrization, the problem is strictly feasible, and one can use a SDP system to obtain $y''$ such that $F'(y'') = -F'_0 + \sum_i y''_i F'_i \succeq 0$ numerically and $F'(y'')$ maximizes a linear form. If, after rationalizing $y''$, $F'(y'')$ is not found to be positive semidefinite (because it has some slightly negative eigenvalues), one can consider a barycenter $F''(\epsilon y' + (1-\epsilon) y'')$ for small~$\epsilon$.
\end{comment}

An alternative to finding some kernel vectors of a single matrix would be to compute several floating-point matrices, for instance obtained by SDP solving with optimization in multiple directions, and find common kernel vectors using LLL.

\subsection{Sub-algorithms and Implementation}
\label{part:implementation}
The reduction from the problem expressed in Eq.~\ref{eqn:sos-equation} to SDP with rational solutions was implemented in \soft{Sage}.%
\footnote{\soft{Sage} is a computer algebra system implemented using the \soft{Python} programming language, available under the GNU GPL from \url{http://www.sagemath.org}.}

Solving the systems of linear equations $(S)$ (Sec.~\ref{part:reduction}, over the coefficients of the matrices) and $\ref{eq:dimension_reduction}$, in order to obtain a system $-F_0 + \textrm{vect}(F_1, \dots, F_m)$ of generators of the solution space, is done by echelonizing the equation system (in homogeneous form) in exact arithmetic, then reading the solution off the echelon form. The dimension of the system is quadratic in the number of monomials (on the problems we experimented with, dimensions up to 7900 were found); thus efficient algorithms should be used. In particular, sparse Gaussian elimination in rational arithmetic, which we initially experimented, is not efficient enough; we thus instead use a sparse multi-modular algorithm~\cite[ch.~7]{Stein_modular07} from \soft{LinBox}%
\footnote{\soft{LinBox} is a library for exact linear arithmetic, used by \soft{Sage} for certain operations. \url{http://www.linalg.org/}}%
. Multi-modular methods compute the desired result modulo some prime numbers, and then reconstruct the exact rational values.

One can test whether a symmetric rational matrix $Q$ is positive semidefinite by attempting to convert it into its Gaussian decomposition, and fail once one detects a negative diagonal element, or a nonzero row with a zero diagonal element (Appendix.~\ref{part:gaussian}). We however experimented with three other methods that perform better:
\begin{itemize}
\item Compute the minimal polynomial of $Q$ using a multi-modular algorithm \cite{Adams:2005:SSR:1073884.1073889}. The eigenvalues of $Q$ are its roots; one can test for the presence of negative roots using Descartes' rule of signs. Our experiments seem to show this is the fastest exact method.

\item Compute the characteristic polynomial of $Q$  using a multimodular algorithm \cite{Adams:2005:SSR:1073884.1073889} and do as above. Somewhat slower but more efficient than Gaussian decomposition.

\item Given a basis $B$ of the span of $Q$, compute the Cholesky decomposition of $B^T Q B$ by a numerical method. This decomposition fails if and only if $B^T Q B$ is not positive definite (up to numerical errors), thus succeeds if and only if $Q$ is positive semidefinite (up to numerical errors).

For efficiency, instead of computing the exact basis $B$ of the span of $Q$, we use $B$ from \S\ref{part:algorithm}, whose span includes the span of $Q$. The only risk is that $\ker B \subsetneq \ker Q$ while $Q$ is positive semidefinite, in which case $B^T Q B$ will have nontrivial nullspace and thus will be rejected by the Cholesky decomposition. This is not a problem in our algorithm: it just means that the procedure for finding kernel vectors by LLL will find vectors in $
\ker Q \setminus \ker B$.

One problem could be that the Cholesky decomposition will incorrectly conclude that $B^T Q B$ is not positive definite, while it is but has very small positive eigenvalues. In this case, our algorithm may then find kernel vectors that are not really kernel vectors, leading to an overconstrained system and possibly loss of completeness. We have not encountered such cases.

Another problem could be that a Cholesky decomposition is obtained from a matrix not positive semidefinite, due to extremely bad numerical behavior. At worst, this will lead to rejection of the witness when the allegedly semidefinite positive matrices get converted to sums of squares, at the end of the algorithm.
\end{itemize}

Numerical SDP solving is performed using DSDP%
\footnote{DSDP is a sdp tool available from \url{http://www.mcs.anl.gov/DSDP/}}
\cite{DSDP_manual,Benson:2008:ADS}, communicating using text files.
LLL reduction is performed by \soft{fpLLL}.%
\footnote{\soft{fpLLL} is a LLL library from Damien Stehl\'e et al., available from \url{http://perso.ens-lyon.fr/damien.stehle/}}
Least square projection is performed using \soft{Lapack}'s DGELS.

The implementation is available from the first author's Web page (\url{http://bit.ly/fBNLhR} and \url{http://bit.ly/gPXNF8}).

\subsection{Preliminary Reductions}
The more coefficients to find there are, the higher the dimension is, the longer computation times grow and the more likely numerical problems become. Thus, any cheap technique that reduces the search space is welcome.

If one looks for witnesses for problems involving only homogeneous polynomials, then one can look for witnesses built out of a homogeneous basis of monomials (this technique is implemented in our positivity checker).

One could also make use of symmetries inside the problem. For instance, if one looks for a nonnegativity witness $P=N/D$ of a polynomial $P$, and $P$ is symmetric (that is, there exists a substitution group $\Sigma$ for the variables of $P$ such that $P.\sigma = P$ for $\sigma \in \Sigma$), then one may reduce the search to symmetric $N$ and $D$. If $P=N/D$ is a witness, then $DP=N$ thus for any $\sigma$, $(D.\sigma) P = (N. \sigma)$ and thus $(\sum_\sigma D.\sigma) P = (\sum_\sigma N.\sigma)$, thus $D' = \sum_\sigma D.\sigma$ and $N' = \sum_\sigma N.\sigma$ constitute a symmetric nonnegativity witness.

\section{Examples}\label{part:examples}
The following system of inequalities has no solution (neither \soft{Redlog} nor \soft{QepCad} nor \soft{Mathematica}~5 can prove it; \soft{Mathematica}~7 can):
\begin{equation}\left\{\begin{array}{l}
P_1 = x^3 + x y + 3 y^2 + z + 1 \geq 0\\
P_2 = 5 z^3 - 2 y^2 + x + 2 \geq 0 \quad P_3 = x^2 + y - z \geq 0\\
P_4 = -5 x^2 z^3 - 50 x y z^3 - 125 y^2 z^3  + 2 x^2 y^2 + 20 x y^3 + 50 y^4 - 2 x^3\\
\quad  - 10 x^2 y - 25 x y^2 - 15 z^3 - 4 x^2 - 21 x y - 47 y^2 - 3 x - y - 8 \geq 0
\end{array}\right.
\end{equation}

This system was concocted by choosing $P_1,P_2,P_3$ somewhat haphazardly and then $P_4 = -(P_1 + (3+(x+5y)^2) P_2 + P_3 + 1 + x^2)$, which guaranteed the system had no solution. The initial 130 constraints yield a search space of dimension 145, and after four round of numeric solving one gets an unsatisfiability witness (sums of squares $Q_j$ such that $\sum_{j=1}^4 P_j Q_j + Q_5 = 0$). Total computation time was 4.4~s. Even though there existed a simple solution (note the above formula for $P_4$), our algorithm provided a lengthy one, with large coefficients (and thus unfit for inclusion here).

\begin{figure}[tpb]
\begin{minipage}{\columnwidth}
\sloppy\scriptsize\raggedright
$Q_1 = 8006878 A_1 ^2
   + 29138091   A_2 ^2
   + 25619868453870\allowbreak /4003439 A_3 ^2
   + 14025608   A_4 ^2
   + 14385502   A_5 ^2
   \allowbreak + 85108577038951965167\allowbreak / 12809934226935   A_6 ^2$\\
$Q_2 = 8006878   B_1^2
   + 25616453   B_2 ^2
   + 108749058736871\allowbreak/4003439   B_3 ^2
   + 161490847987681\allowbreak/25616453   B_4 ^2
   + 7272614   B_5 ^2
   + 37419351   B_6 ^2
   + 13078817768190\allowbreak/3636307   B_7^2
   + 71344030945385471151\allowbreak /  15535579819553   B_8 ^2
   + 539969700325922707586\allowbreak /161490847987681   B_9 ^2
   + 41728880843834\allowbreak / 12473117   B_{10}^2
   + 131008857208463018914\allowbreak /62593321265751   B_11^2$,
where\\
$A_1 = -1147341\allowbreak /4003439 x_1^2 x_3 - 318460\allowbreak /4003439 x_2^2 x_3 + x_3^3$
$A_2 = x_2 x_3^2$
$A_3 = -4216114037644\allowbreak /12809934226935 x_1^2 x_3 + x_2^2 x_3$
$A_4 = x_1 x_3^2$,
$A_5 = x_1 x_2 x_3$,
$A_6 = x_1^2 x_3$
and
$B_1 =  -1102857\allowbreak /4003439 x_1^4 x_2 x_3 - 5464251\allowbreak /4003439 x_1^2 x_2 x_3^3 + 2563669\allowbreak /4003439 x_2^3 x_3^3 + x_2 x_3^5$,
$B_2 = -9223081\allowbreak /25616453 x_1^4 x_3^2 - 18326919\allowbreak /25616453 x_1^2 x_2^2 x_3^2 + 1933547\allowbreak /25616453 x_2^4 x_3^2 + x_2^2 x_3^4$,
$B_3 = -2617184886847\allowbreak /15535579819553 x_1^4 x_2 x_3 - 12918394932706\allowbreak /15535579819553 x_1^2 x_2 x_3^3 + x_2^3 x_3^3$,
$B_4 =  -26028972147097\allowbreak /161490847987681 x_1^4 x_3^2 - 135461875840584\allowbreak /161490847987681 x_1^2 x_2^2 x_3^2 + x_2^4 x_3^2$,
$B_5 = -2333331\allowbreak /3636307 x_1^3 x_2 x_3^2 - 1302976\allowbreak /3636307 x_1 x_2^3 x_3^2 + x_1 x_2 x_3^4$,
$B_6 = -11582471\allowbreak /37419351 x_1^5 x_3 - 12629854\allowbreak /37419351 x_1^3 x_2^2 x_3 - 4402342\allowbreak /12473117 x_1^3 x_3^3 + x_1 x_2^2 x_3^3$,
$B_7 = -x_1^3 x_2 x_3^2 + x_1 x_2^3 x_3^2$,
$B_8 =  -x_1^4 x_2 x_3 + x_1^2 x_2 x_3^3$,
$B_9 = -x_1^4 x_3^2 + x_1^2 x_2^2 x_3^2$,
$B_{10} =  -17362252580967\allowbreak /20864440421917 x_1^5 x_3 \allowbreak - 3502187840950\allowbreak /20864440421917 x_1^3 x_2^2 x_3 + x_1^3 x_3^3$,
$B_{11} =  -x_1^5 x_3 + x_1^3 x_2^2 x_3$.
\end{minipage}

\caption{Motzkin's polynomial $M$ (Eq.~\ref{eqn:motzkin}) as $Q_2/Q_1$.}
\label{fig:motzkin_decomposition}
\end{figure}

Motzkin's polynomial $M$ (Eq.~\ref{eqn:motzkin}) cannot be expressed as a sum of squares, but it can be expressed as a quotient of two sums of squares. We solved $M. Q_1 - Q_2 = 0$ for sums of squares $Q_1$ and $Q_2$ built from homogeneous monomials of respective total degrees $3$ and $6$ --- lesser degrees yield no solutions (Fig.~\ref{fig:motzkin_decomposition}). The equality relation over the polynomials yields 66 constraints over the matrix coefficients and a search space of dimension 186. Four cycles of SDP programming and LLL are then needed, total computation time was 4.1~s.

We exhibited witnesses that each of the 8 semidefinite positive forms listed by \cite{Reznick96_Hil17}, which are not sums of squares of polynomials, are quotients of sums of squares (Motzkin's $M$, Robinson's $R$ and $f$, Choi and Lam's $F$, $Q$, $S$, $H$ and Schm\"udgen's $q$). These examples include polynomials with up to 6 variables and search spaces up to dimension 1155.  We did likewise with \textit{delzell}, \textit{laxlax} and \textit{leepstarr2} from \cite{Kaltofen_et_al_JSC09}. The maximal computation time was 7'.

We then converted these witnesses into Coq proofs of nonnegativity using a simple \soft{Sage} script. These proofs use the \texttt{Ring} tactic, which checks for polynomial identity. Most proofs run within a few seconds, though \emph{laxlax} takes 7'39'' and Robinson's $f$ 5'07''; the witness for \emph{leepstarr2} is too large for the parser.
We also exhibited a witness that the \textsf{Vor1} polynomial cited by \cite{Safey_El_Din_ISSAC08} is a sum of squares.

John Harrison kindly provided us with a collection of 14 problems that his system \cite{harrison-sos} could not find witnesses for. These problems generally have the form $P_1 \geq 0 \land \dots \land P_n \geq 0 \Rightarrow R \geq 0$. In order to prove such implication, we looked for witnesses consisting of sums of squares $(Q_1, \dots, Q_n, Q_R)$,such that $\sum_j Q_j P_j + Q_R R = 0$ with $Q_R \neq 0$, and thus $R = \frac{\sum_j Q_j P_j}{Q_R}$. In some cases, it was necessary to use the products $\prod_i P_i^{w_i}$ for $\ve{w} \in \{0,1\}^n$ instead of the $P_i$. We could find witnesses for all those problems,%
\footnote{A 7z archive is given at \url{http://bit.ly/hM7HW3}}.
though for some of them, the witnesses are very large, taking up megabytes. Since these searches were done without making use of symmetries in the problem, it is possible that more clever techniques could find smaller witnesses.

\section{Conclusion and further works}
We have described a method for solving SDP problems in rational arithmetic. This method can be used to solve sums-of-squares problems even in geometrically degenerate cases. We illustrated this method with applications to proving the nonnegativity of polynomials, or the unsatisfiability of systems of polynomial (in) equalities. The method then provides easily checkable proof witnesses, in the sense that checking the witness only entails performing polynomial arithmetic and applying a few simple mathematical lemmas. We have implemented the conversion of nonnegativeness witnesses to Coq proofs. A more ambitious implementation, mapping Coq real arithmetic proofs goals to Positivstellensatz problems through the \texttt{Psatz} tactic from the \soft{MicroMega} package \cite{Besson:2006:FRA:1789277.1789281}, then mapping Positivstellensatz witnesses back to proofs, is underway.

One weakness of the method is that it tends to provide ``unnatural'' witnesses --- they tend to have very large coefficients. These are machine-checkable but provide little insights to the reader. An alternative would be to provide the matrices and some additional data (such as their minimal polynomial) and have the checker verify that they are semidefinite positive; but this requires formally proving, once and for all, some non-trivial results on polynomials, symmetric matrices and eigenvalues (e.g. the Cayley-Hamilton theorem), as well as possibly performing costly computations, e.g. evaluating a matrix polynomial.

A more serious limitation for proofs of unsatisfiability is the very high cost of application of the Positivstellensatz. There is the exponential number of polynomials to consider, and the unknown number of monomials. It would be very interesting if there could be some simple results, similar to the Newton polytope approach, for reducing the dimension of the search space or the number of polynomials to consider. Another question is whether it is possible to define SDP problems from Positivstellensatz equations for which the spectrahedron has rational points only at its relative boundary.

While our method performed well on examples, and is guaranteed to provide a correct answer if it provides one, we have supplied no completeness proof --- that is, we have not proved that it necessarily provides a solution if there is one. This is due to the use of floating-point computations. One appreciable result would be that a solution should be found under the assumption that floating-point computations are precise up to $\epsilon$, for a value of $\epsilon$ and the various scaling factors in the algorithm depending on the values in the problem or the solution.

\begin{comment}
A desirable property of an unsatisfiability proof procedure is that is uses only a minimal set of hypotheses (or, better, a set of hypotheses of minimal cardinality). This would correspond to finding a block diagonal solution matrix $\hat{Q}$ with a maximal set of zero diagonal blocks (or a set of maximal cardinal). Heuristically, one would expect that it helps to look for $\hat{Q}$ of minimal rank within the solution set. This, however, would involve finding solutions on the relative boundary of the solution set, while our method finds solution in the relative interior of that set. In any case, it is possible to use a minimization procedure for the set of hypotheses \cite[\textsf{min} function]{BradleyManna08}.
\end{comment}

It seems possible to combine our reduction method based on LLL with the Newton iterations suggested by \cite{Kaltofen_et_al_JSC09,Kaltofen_et_al_ISSAC08}, as an improvement over their strategy for detection of useless monomials and reduction of the search space. Again, further experiment is needed.

\begin{comment}
\paragraph{Acknowledgements}
We wish to thank Alexis Bernadet for discussions and ideas regarding the degeneracy problem, Mohab Safey El Din, Fr\'d\'eric Besson and Assia Mahboubi for their helpful comments and ideas, Cl\'ement Pernet for explanations on multi-modular methods, and Fran\c{c}ois Morain for his idea of trying LLL.
\end{comment}

\printbibliography

\appendix

\section{Gaussian Reduction and Positive Semidefiniteness}
\label{part:conversion_sos}
\label{part:gaussian}
An algorithm for transforming a semidefinite positive matrix into a ``sum of squares'' form, also known as Gaussian reduction:
\begin{lstlisting}
for i:=1 to n do
begin
  if m[i, i] < 0 then
    throw non_positive_semidefinite
  if m[i, i] = 0 then
    if m.row(i) <> 0
      throw non_positive_semidefinite
  else
    begin
      v := m.row(i) / m[i, i]
      output.append(m[i, i], v)
      m := m - m[i, i] * v.transpose() * v
    end  
end
\end{lstlisting}

Suppose that the entrance of iteration $i$, $m[i\dots n, i\dots n]$ is positive semidefinite. If $m_{i,i}=0$, then the $i$th base vector is in the isotropic cone of the matrix, thus of its kernel, and the row $i$ must be zero. Otherwise, $m_{i,i} > 0$. By adding $\epsilon$ to the diagonal of the matrix, we would have a positive definite matrix and thus the output of the loop iteration would also be positive definite, as above. By $\epsilon \rightarrow 0$ and the fact that the set of positive semidefinite matrices is topologically closed, then the output of the loop iteration is also positive semidefinite.

The \lstinline!output! variable is then a list of couples $(c_i,v_i)$ such that $c_i > 0$ and the original matrix $m$ is equal to $\sum_i c_i \transpose{v_i} v_i$ (with $v_i$ row vectors). Otherwise said, for any row vector $u$,
$u m \transpose{u} = \sum_i c_i \langle u, v_i \rangle^2$.
%(We could do away with the $c_i$ coefficients, by replacing $v_i$ by $\sqrt{c_i} v_i$, but then the results would be algebraic, not necessarily rational.)

%%% Local Variables: 
%%% mode: latex
%%% TeX-master: psatz_witnesses_itp2011_article.tex
%%% End: 

\end{document}